\documentclass{amsart}

\usepackage[english]{babel}
\usepackage{esint}
\usepackage{amsfonts,amssymb,amsmath}
\usepackage{setspace}
\usepackage{centernot}
\usepackage{mathrsfs} 
\usepackage{pst-plot}
\usepackage{auto-pst-pdf}

\def \erre {{\mathbb {R}}}

\def \elle {{\mathscr {L}}}
\def \effe {{\mathscr {F}}}
\def \picorsivo {{\mathscr {P}}}
\def \esse {{\mathscr {S}}}
\def \acca{{\mathcal {H}}}

\def\erren{{\erre^{ {n} }}}

\newcommand{\tende}{\rightarrow}
\newcommand{\ttende}{\longrightarrow}
\newcommand{\enne} {\mathbb{N}}
\newcommand{\frecciaf} {\longmapsto}

\newcommand\partiali{\partial_{x_i}}
\newcommand\partialj{\partial_{x_j}}

\newtheorem{theorem}{Theorem}

\newtheorem{corollary}[theorem]{Corollary}
\newtheorem{lemma}[theorem]{Lemma}
\theoremstyle{remark}
\newtheorem{remark}[theorem]{Remark}
\theoremstyle{definition}

\newcounter{tmp}

\numberwithin{equation}{section}

\title[Harnack inequality for hypoelliptic operators]{Harnack inequality for 
hypoelliptic second order partial differential operators}

\author{Alessia E. Kogoj}
\address {Dipartimento di Ingegneria dell'Informazione, Ingegneria Elettrica e Matematica Applicata, Universit\`a degli Studi di Salerno, Via Giovanni Paolo II, 132, IT-84084 Fisciano (SA) - Italy}

\email{alessia.kogoj@gmail.com}

\author{Sergio Polidoro}
\address {Dipartimento di Scienze Fisiche, Informatiche e Matematiche - 
Universit\`{a} di Modena e Reggio Emilia - 
via Campi 213/b,  IT-41125 Modena - Italy}
\email{sergio.polidoro@unimore.it}


\begin{document}


\begin{abstract} 

We consider nonnegative solutions $u:\Omega\longrightarrow \mathbb{R}$ of second order hypoelliptic  equations 
\begin{equation*} \mathscr{L} u(x) =\sum_{i,j=1}^n  \partial_{x_i} \left(a_{ij}(x)\partial_{x_j} u(x) \right)  + \sum_{i=1}^n b_i(x) \partial_{x_i} u(x) =0,
\end{equation*}
where $\Omega$ is a bounded open subset of $\mathbb{R}^{n}$ and  $x$ denotes the point of $\Omega$.
For any fixed $x_0 \in \Omega$, we prove a Harnack inequality of this type
$$\sup_K u \le C_K u(x_0)\qquad \forall \ u \ \mbox{ s.t. } \  \mathscr{L} u=0, u\geq 0,$$
where $K$ is any compact subset of the interior of the $\mathscr{L}$-{\it propagation set of $x_0$} and the constant $C_K$ does not depend on $u$. 

\end{abstract}
\maketitle

\section{Introduction}

We consider second order partial differential operators $\elle$ acting on functions $u \in C^2 ( \Omega )$ 
as follows
\begin{equation} \label{operatore} 
  \elle u(x) : = \sum_{i,j=1}^n \partiali \left(a_{ij}(x)\partialj u(x) \right) + 
\sum_{i=1}^n b_i(x) \partiali u(x)  
\end{equation}
for $x$ belonging to any open {\it bounded} subset $\Omega$ of $\erren$. The coefficients $a_{ij}, b_i$ are real 
functions and belong to $C^\infty(\overline \Omega)$ for $1 \le i,j \le n$. Moreover, $A := (a_{ij})$ is a $n \times n$ 
symmetric and non-negative matrix. We also assume 
the following hypotheses:
\begin{itemize}

\item[(H1)]  { \it $\elle - \beta$ and $\elle^*$ are hypoelliptic} for every constant $\beta \ge 0$;
 
\item[(H2)] $\inf_\Omega \ a_{11} >  0.$
\end{itemize} 
We recall that $\elle$ is said hypoelliptic if every distribution $u$ in $\Omega$ such that $\elle u \in C^\infty 
(\Omega)$ is a smooth function. 
We note that condition (H2) ensures that  for every  $x \in \Omega$ there exists $\xi \in \erre^n$ such that $\langle 
A(x) \xi, \xi \rangle > 0$ that is $\elle$ is non-totally degenerate, in accordance with Definition 
5.1 in \cite{bony_1969}.  We can drop condition (H2) if the operator $\widetilde\elle = \partial_{x_{n+1}}^2 + \elle$ acting on $\erre^{n+1}$   satisfies (H1) (see Corollary \ref{corollariosergio}). 

The main result of this paper is the following Harnack inequality for the non-negative solutions of the equation
$\elle u=0$. It obviously applies to the Laplacian and to the heat operators and in these cases it restores the classical elliptic and parabolic Harnack inequalities. 
\begin{theorem}\label{maintheorem} Assume that $\elle$ statisfies (H1) and (H2).  Let $x_0$ in $\Omega$ and let $K$ be any compact set  contained in the interior 
of $\overline {\picorsivo(x_0,\Omega)}$, then there exists a positive constant $C=C(x_0, K, \Omega,\elle)$ such that 
$$
  \sup_K u \le C u(x_0), 
$$
for every non-negative solution $u$ of $\elle u=0$ in $\Omega$.
\end{theorem}
We introduce here the definition of $\elle$-\emph{propagation set} $\picorsivo(x_0,\Omega)$ appearing it the above statement. 
It is the set of all points $x$ reachable from  $x_0$  by a {\it propagation path}:
$$\picorsivo(x_0,\Omega) : = \{ x \in \Omega\, |\, \exists\   \gamma\  \mbox{$\elle$-propagation 
path}, \gamma(0)=x_0, \gamma(T)=x \}. $$

A $\elle$-{\it propagation path} is any absolutely continuous path $\gamma: [0,T] \longrightarrow \Omega$
such that 
\[\gamma '(t)= \sum_{j=1}^n \lambda_j(t) X_j(\gamma(t))+\mu(t)Y(\gamma(t)) \quad\quad
\mbox{ a.e. in } [0,T]\] 
for suitable piecewise constant real functions  $\lambda_1,\ldots,\lambda_n,$ and $\mu$, $\mu\geq 0$. \\ $X_1(x), 
\ldots, X_n(x), Y(x)$ are the vector fields defined in the following way:
\begin{equation} \label{eq-XY}
 X_j(x) := \sum_{i=1}^n a_{ji} (x) \partial_{x_i},\quad j=1, \dots, n, \qquad
 Y(x) := \sum_{i=1}^n b_i(x) \partial_{x_i}.
\end{equation}

As we said at the beginning of the Introduction, our main result can holds also under somehow weaker assumptions on $\elle$.
Only in the following Corollary the hypotheses (H1) and (H2) on $\elle$ are replaced by the assumption that the operator $\widetilde\elle = \partial_{x_{n+1}}^2 + \elle$ in $\erre^{n+1}$ satisfies (H1). Of course if $\widetilde\elle$ satisfies assumption  (H1) then also $\elle$ does.
A simple example of operator satisfying the hypotheses of this Corollary but not the ones of Theorem 1 is $\partial_{x_1}={x_1}^2\partial^2_{x_2}$ in $\erre^2$. 
\begin{corollary} \label{corollariosergio}
If   the operator  $\widetilde\elle = \partial_{x_{n+1}}^2 + \elle$ acting on   $\Omega \times \erre$   satisfies (H1)
then for every $x_0$ in $\Omega$ and for every compact set $K$ contained in the
interior of the $\elle$-propagation set $\overline {\picorsivo(x_0,\Omega)}$,  there exists a positive constant $C=C(x_0, K, \Omega,\elle)$ such
that
$$
 \sup_K u \le C u(x_0),
$$
for every non-negative solution $u$ of $\elle u=0$ in $\Omega$.
\end{corollary}

The notion of propagation set $\picorsivo(x_0,\Omega)$ has been introduced by Amano in his work on maximum principle (see 
\cite[Theorem 2]{amano_max_1979}). In our case it reads as follows. 

{\it Assume that $u$ is a (smooth) solution of $\elle u=0$ in $\Omega.$ If $u$ attains its maximum at a point $x_0$ in 
$\Omega$, then $u\equiv u(x_0)$ in $\overline{\picorsivo(x_0,\Omega)}.$}

The Amano maximum principle is a crucial tool to prove Theorem \ref{maintheorem} using the {abstract Harnack inequality} 
of the Potential Theory (\cite[Proposition 6.1.5]{CC}). We observe that Harnack inequalities based on results of Potential Theory 
were proved in  \cite[Theorem 4.2]{bonfiglioli_lanconelli_uguzzoni_uniform}  for heat equations on Carnot groups and in 
\cite[Theorem 1.1]{cinti_nystrom_polidoro} and in \cite[Theorem 5.2]{cinti_menozzi_polidoro} for more general evolution equations. 
The class of hypoelliptic operators considered in  \cite{cinti_nystrom_polidoro,cinti_menozzi_polidoro} is 
\begin{equation} \label{eq-evol}
  \sum_{j=1}^m \widetilde X^2_j(y)   +\widetilde X_0(y) - \partial_t, 
\end{equation}
where $\widetilde X_j$ are smooth vector fields on $\erre^N$ and $(y,t)$ denoted the point of  any subset of 
$\erre^{N+1}$. We explicitly note that operator \eqref{eq-evol} is a particular example of the operators 
\eqref{operatore}, with respect to the variable $x =(y,t)$.  In both papers  the operators are assumed left translation 
invariant  w.r.t. a Lie group in $\erre^{N+1}$ and endowed with a {\it global} fundamental solution.  We point out that 
the use of the fundamental solution is a key step in verifying the {\it separation axiom} of the axiomatic Potential 
Theory. 

The approach used in this note allows us to prove the validity of the separation axiom in Section \ref{axioms} without
requiring  the existence of any global fundamental solution on every bounded open set.  We rely  only on  
hypoellipticity,  on non total degeneracy of 
$\elle$ and the following maximum principle due to Picone that we recall for the sake of completeness.

{\it     Let  $V$ be any open (bounded) subset of  $\Omega$. Assume that exists a function  $w: V \rightarrow \erre$ such that  $\elle w<0$ in  $V$ and $\inf_{V} w>0.$  
Then for every  $u\in C^2 (V)$ such that \[\elle u\geq 0\quad\mbox { in } 
V,\quad\limsup_{x\tende \xi}u(x) \le 0\quad \forall \xi \in\partial V, \]  we have $u\le 0$ in  $V$.}

In our case the existence of a function $w$ follows from (H2) and from the smoothness of the coefficients. Indeed, under these assumptions, we can choose two  positive real constants $M$ and 
$\lambda$ such that  the function \begin{equation}\label{eqw} w(x)=w(x_1,\ldots, x_N)= M - e^{\lambda x_1}\end{equation} has the required properties.

This paper is organized as follows. In Section 2  all the notions and results from  Potential Theory that we need are 
briefly recalled.
In Section 3 we show that the set of the solutions $u$ of $\elle u=0$ in $\Omega$  satisfies the axioms of the Doob  Potential Theory.
In Section 4 we prove that the $\elle$-propagation set of $x_0$  is a subset of the smallest absorbent set 
containing  $x_0$. In this way we derive the Harnack inequality for the non-negative solutions $u$ of $\elle u =0$. In Section 5  the propagation sets of some meaningful operators are studied. 
 In particular we focus on the following operators:
$\partial_{{x}_1}^2 + x_1 \partial_{x_2}$ in $\erre^2$ and $\partial_{x_1}^2 + \sin(x_1) \partial_{x_2} + \cos(x_1) \partial_{x_3}$ in $\erre^3$, and we show that  the geometry of the relevant Harnack inequality may appear either  of \emph{parabolic} or \emph{elliptic}-type, depending on the choice of $\Omega$, even if both operators are \emph{parabolic}.

\section{Some recalls from Potential Theory}

We recall some definitions and results of the Potential Theory that we need to prove our Harnack inequality.
For a detailed description of the general theory of {\it harmonic spaces} we refer to \cite[chapter 6]{BLU}, \cite{CC} and to \cite{bauer}.

\subsection{Sheafs of functions and harmonic sheafs in $\Omega$} 
\mbox{}\\
Let $V$ be any open subset of $\Omega$. We denote by $\overline\erre$ the set $\mbox{$\erre \cup \{\infty, -\infty\}$ 
and by $\overline\erre^V$}$ the set of functions $\mbox{$u: V \longrightarrow \overline\erre$}$.  Moreover $C(V,\erre)$  is the vector space of real continuous functions defined on $V$.

A map 
$$\effe : V \mapsto \effe(V) \subseteq \overline\erre^V$$ 
is a {\it sheaf of functions} in $\Omega$ if
\begin{itemize}

\item[$(i)$] $V_1, V_2 \subseteq \Omega$, $V_1 \subseteq V_2$, $ u\in \effe(V_2)$ $\implies$ $u|_{V_1} \in 
\effe({V_1})$;

\item[$(ii)$] $V_\alpha  \subseteq \Omega\  \forall \alpha \in \mathcal{A},  u: \bigcup_{\alpha \in A} V_\alpha 
\longrightarrow \overline\erre$, $u|_{V_\alpha} \in \effe({V_\alpha}) \implies u\in \effe({ \bigcup_{\alpha\in 
\mathcal{A}}}V_\alpha).$ 

\end{itemize}

\noindent When $\effe(V)$ is a linear subspace of $C(V,\erre)$ for every $V\subseteq\Omega$, we say that the sheaf of 
functions $\effe$ on $V$ is {\it harmonic} and we denote it $\acca(\Omega).$
\subsection{Regular open sets, harmonic measures and absorbent sets} 
\mbox{}\\
Let $\acca$ be a harmonic sheaf on $\Omega$. We say that an open set $V\subseteq \Omega$ is
{\it regular}
if:
\begin{itemize}

\item[$(i)$] $\overline{V}\subseteq \Omega$ is compact and $\partial V \neq \emptyset$;

\item[$(ii)$] for every continuous function $\varphi : \partial V \longrightarrow \erre$, there exists a unique 
function 
in $\acca(V)$, that we denote by $h_\varphi^V$, such   that 
$h_\varphi^V(x) \xrightarrow[x\tende \xi ]{} \varphi (\xi)$ for every $\xi \in \partial V;$
\item[$(iii)$] if $\varphi\geq 0$ then $h_\varphi^V \geq 0.$ 

\end{itemize}
From $(ii)$ and $(iii)$ it follows that, for every  regular set $V$ and for every $x \in V$, the map 
$$C(\partial V) \ni \varphi \longmapsto h_\varphi^V(x) \in \erre$$ is linear and positive.
Thus, the Riesz representation theorem (see e.g. \cite{rudin}), 
implies that, for every  regular set $V$ and for every $x \in V$, there exists a {\it regular Borel measure}, that we 
denote by $\mu^V_x$, supported in $\partial V$, 
such that 
$$h_\varphi^V (x) = \int_{\partial V} \varphi(y) \ d \mu^V_x(y) \qquad \forall\ \varphi\in C(\partial V).$$ 
The measure $\mu^V_x$ is called the {\it harmonic measure} related to $V$ and $x$.

Now, let $A$ be a closed subset of $\Omega$ . We say that $A$ is {\it absorbent} if  it contains the supports of all the 
harmonic measures  related to its points.  More precisely, $$\mbox{for every $x \in A$ and every regular set $V$ containing $x$, 
$\mathrm{supp\ } \mu_{x}^V\subseteq A$.}$$

If $x_0 \in\Omega $, we define $\Omega_{x_0}$ as the smallest absorbent set containing  $x_0$:

$$\Omega_{x_0}:=  \bigcap_{\substack {A \mbox{\tiny\  absorbent} \\ A\ni x_0}} A.$$

\subsection{Superharmonic functions} \mbox{}\\
A function $u: \Omega \ttende ]-\infty, \infty]$ is called {\it superharmonic in $\Omega$} if 
 \begin{itemize}

\item[$(i)$] $u$ is lower semi-continuous;

\item[$(ii)$] for every regular set $V$, $\overline V\subseteq \Omega$, and  for every $\varphi \in C( \partial V, 
\erre)$, 
$\varphi \le u|_{\partial \Omega} $,  it follows  $ u \geq h_\varphi^V$ in $V;$

\item[$(iii)$] the set $\{ x\in \Omega \ | \ u(x) < \infty \}$ is dense in $\Omega$.

\end{itemize}
We denote by $\esse(\Omega)$ the family  of the superharmonic functions on $\Omega$.

By the maximum principle, we have that every function $u\in C^2(\Omega)$ such that $\elle u \le 0$ in $\Omega$ is 
superharmonic (see \cite[Proposition 7.2.5]{BLU}).

\subsection{Doob harmonic spaces and Harnack inequality} \label{axioms}
\mbox{}\\
We say that a harmonic sheaf $\acca(\Omega)$ is a {\it Doob harmonic space} if the following axioms are satisfied.

 \begin{itemize}

\item[(A1)]  { \it Positivity axiom:}\\
For every $x\in \Omega$, there exists an open set $V\ni z$ and a function $u\in \acca(V)$ such that $u(x)>0$.

\item[(A2)]  { \it Doob convergence axiom:}\\
Let $(u_n)_{n\in\enne}$ be a monotone increasing sequence in $\acca(\Omega)$ and let  \\ $u:= \sup_{n\in\enne} u_n.$ If  
 the set $\{ x\in \Omega \ | \ u(x) < \infty \}$ is dense in $\Omega$, then $u\in \acca(\Omega).$

\item[(A3)]  { \it Regularity axiom:}\\   There is a basis of the euclidean topology of $\Omega$ formed by  regular 
sets.

\item[(A4)]  { \it Separation axiom:}\\  $\esse(\Omega)$ separates the points of $\Omega$ in this sense:
for every $y$ and $z$ in $\Omega$, $y\neq z$, there exist two non-negative  functions $u$ and $v$ in $\esse(\Omega)$ such that 
$u(y)v(z)\neq u(z)v(y).$
\end{itemize}

We close the section recalling that in this setting the {\it abstract Harnack inequality} from the Parabolic Potential 
Theory holds \cite[Proposition 6.1.5
]{CC}.
\begingroup
\setcounter{tmp}{\value{theorem}}
\setcounter{theorem}{0} 
\renewcommand\thetheorem{\Alph{theorem}}
\begin{theorem} \label{abstract_harnack} Let  $(\Omega, \acca)$ be a {\it Doob harmonic space}, $x_0\in \Omega$ and let 
$K$ be a compact set contained in the interior of $\Omega_{x_0}$,  the smallest absorbent set containing $x_0$.
Then there exists a positive constant $C=C(x_0, K, \Omega)$ such that 
$$\sup_K u \le C u(x_0)\qquad\forall u\in \acca(\Omega), u\geq0.$$
\end{theorem}
\endgroup

\section{The harmonic space of the solutions of  $\elle u=0$ }
We show that the set of the solutions of the equation $\elle u=0$ is a Doob harmonic space in $\Omega.$ 
For every $V\subseteq \Omega$ we consider the harmonic sheaf 
$$ \erren \supseteq V \frecciaf \acca(V)$$

where 
\begin{equation*} \acca(V)= \{ u\in C^\infty(V) \ | \ \elle u=0 \}
\end{equation*}
and  $\elle$ is the operator  \eqref{operatore}.

The {\it positivity axiom } (A1) is plainly verified. Indeed every constant function belongs to $\acca(\Omega)$.

(A2) is a consequence of a weak Harnack inequality  due to Bony (see \cite[Theoreme 7.1]{bony_1969}); see also 
\cite[Proposition 7.4]{kogoj_lanconelli_2004}).

 (A3), i.e. the existence of a basis of the euclidean topology of $\Omega$ formed by  regular 
sets,  can be proved as in  \cite[Corollarie 5.2]{bony_1969}, see also \cite[Proposition 7.1.5]{BLU}.  We stress that the tools used in its 
proof are only the hypoellipticity, the non totally degeneracy of the operator $\elle$  and the classical 
Picone  Maximum Principle.

Now we are left to verify the {\it separation axiom } (A4).  As in our setting the constant are superharmonic functions, we  need to prove that
\begin{equation}\label{separazione} \forall\ y, z \in \Omega, y\neq z,  \exists\  u \in \esse(\Omega),\ u\geq 0,  \mbox{ 
such that } u(y)\neq u(z).\end{equation}

Now, let $y=(y_1,\ldots, y_n)$ and $z=(z_1,\ldots, z_n)$  be two different points in $\Omega$.

We observe that the  function $ w(x)=w(x_1,\ldots, x_N)= M- e^{\lambda x_1}$, as in \eqref{eqw}, for suitable real positive constants $\lambda$ and $M$,  is non-negative and  $\elle w(x) <0$  for every $x \in \Omega$, hence $w\in \esse(\Omega)$.

If  $y_1\neq z_1$, we can choose $u(x)=w(x)$ to separate $y$ and $z$  and we are done.

If  $y_1= z_1$, we set  $u(x)=  |x-y|^2 + w(x)$. Also in this case, for suitable $\lambda$ and $M$, $u$ is non-negative, $u \in C^2(\Omega)$ and $\elle u(x) = \elle (|x-y|^2) + \elle(w(x))< 0$  in  $\Omega$. Moreover $u(y) - u(z) = |z-y|^2$, so \eqref{separazione} is  satisfied.

\section{Propagation sets and Harnack inequality}

Let $X_1(x), \ldots, X_n(x), Y(x)$ be the vector fields defined in the following way:
\begin{equation*}\begin{split}  &
 X_i(x)= \sum_{j=1}^n a_{ij}(x)  \partial_{x_j},\qquad 1 \le i \le n,
\\ & Y(x) = \sum_{i=1}^n  b_i(x) 
\partial_{x_i}.
\end{split}
\end{equation*} 

We recall  that a $\elle$-{\it propagation path}   is any absolutely continuous path $\gamma: [0,T] \longrightarrow \Omega$
such that 
\[\gamma '(t)= \sum_{j=1}^n \lambda_j(t) X_j(\gamma(t))+\mu(t)Y(\gamma(t)) \quad\quad
\mbox{ a.e. in } [0,T]\] for suitable piecewise constant
real functions  $\lambda_1,\ldots,\lambda_n,$ and $\mu$ with $\mu\geq 0$.

For a point $x_0$ in $\Omega$, we define the $\elle$-{\it propagation set} as the set of all points $x$ such that $x$ and 
$x_0$ can be connected by a propagation path, running from $x_0$ to $x$:

$$\picorsivo(x_0,\Omega) : = \{ x \in \Omega\, |\, \exists\   \gamma: [0,T] \tende  \Omega, \gamma\  \mbox{$\elle$-propagation 
path}, \gamma(0)=x_0, \gamma(T)= x  \}. $$
Proceeding as in \cite[Lemma 5.8]{cinti_menozzi_polidoro}, we prove  now that the $\elle$-propagation set of $x_0$ is a 
subset of every absorbent set containing $x_0.$ This Lemma, based on  the maximum propagation principle, is a key lemma 
in order to get our Harnack inequality so 
we prefer to give here its detailed proof. 
\begin{lemma} \label{lemma2} For every $x_0$ in $\Omega$, $\picorsivo(x_0,\Omega) \subseteq \Omega_{x_0}.$
\end{lemma}
\begin{proof} By contradiction,  suppose $x \in \picorsivo(x_0,\Omega)$ and $ x \notin \Omega_{x_0}.$  
There exists an absolutely continuous path $\gamma$ connecting $x_0$ and $ x $:
$$\gamma : [0,T] \longrightarrow \Omega,\qquad \gamma(0)=x_0, \qquad \gamma(T)=x.$$
As $ \Omega_{x_0}$ is a  subset closed  in $\Omega$ and $\gamma$ is continuous, there will be a time $t_1$ such that 
$\gamma(t_1)=x_1 \in \Omega_{x_0}$ and $\gamma(t)\notin \Omega_{x_0}$ when $t$ is in $ ] t_1,T].$

Let's take a regular open set $V$ containing $x_1$. There will be $t_2 \in ] t_1,T]$ such that $x_2=\gamma(t_2) \in 
\partial V$.
From what we wrote before, $x_2$ does not belong to $\Omega_{x_0}$. 

Take  now a neighborhood  of  $x_2$, $U$  such that $U \cap \partial V \subseteq \Omega \backslash \Omega_{x_0}$ and 
consider a function $\varphi$ defined on $\partial V$ such that $\varphi$ is strictly positive in  $U \cap \partial V$ 
and $0$ otherwise. 

$$h_\varphi^V(x_1) = \int_{\partial V} \varphi( y)\ d \mu_{x_1}^V (y) = \int_{U \cap \partial V } \varphi( \zeta)\ d 
\mu_{x_1}^V (\zeta) =0,$$
because $x_1$ is in $\Omega_{x_0}$ and $\mathrm{supp\ } \mu_{x_1}^{V} \subseteq  \Omega_{x_0}$ for every regular set 
$V$.
But  $h_\varphi^V$  is nonnegative and it would attain its minimum at  $ x_1$.
From Amano minimum propagation principle \cite[Theorem 2]{amano_max_1979},  it would follow  that 
$$h_\varphi^V(\gamma(t))=0\qquad\forall \ t \in ]t_1,t_2[.$$
In conclusion, we would have that 
 $$h_\varphi^V(x) \xrightarrow[x\tende x_2 ]{} \varphi (x_2)>0,$$ and
 $$h_\varphi^V(\gamma(t)) \xrightarrow[t\tende t_2^{-} ]{} h_\varphi^V(\gamma(t_2))=0$$
that  is a contraddiction.
\end{proof}

We are now ready to give the proofs of  our main results.

\begin{proof}[Proof of Theorem \ref{maintheorem}]
Let $x_0$ in $\Omega$ and let $K$ be a compact set 
 contained in the interior of $\overline{\picorsivo(x_0,\Omega)}$. 
 As  $\Omega_{x_0}$  is a closed subset of $\Omega$, Lemma \ref{lemma2} implies that $\overline{\picorsivo(x_0,\Omega)} 
\subseteq \Omega_{x_0}$.  On the other hand,  as we showed in Section 3,  the set of the solutions of the equation 
$\elle u=0$ is a Doob harmonic space in $\Omega$. Then, by  Theorem \ref{abstract_harnack},  there exists a positive 
constant $C=C(x_0, K, \Omega,\elle)$ such that 
$$\sup_K u \le C u(x_0)\,$$
for every non-negative solution $u$ of  $\elle u=0$ in $\Omega$.
\end{proof}

\begin{proof}[Proof of Corollary \ref{corollariosergio}]
 We set $\widetilde x : = (x,x_{n+1})$, $\widetilde\Omega:= \Omega \times ]-1,1[$, $\widetilde K: = K \times [- \frac{1}{2}, \frac{1}{2}]$ and $\widetilde u(\widetilde x): = u(x)$ for every $\widetilde x\in \widetilde\Omega$. 
We observe that the $\widetilde \elle$-propagation set of $(x_0,0)$, $\widetilde \picorsivo_{(x_0,0)}(\widetilde\Omega)$, equals 
 $\picorsivo_{x_0}(\Omega) \times ]-1,1[$. Then $K\subseteq \mathrm{int} \picorsivo_{x_0}(\Omega)$ if and only if 
 $\widetilde K\subseteq \mathrm{int} \widetilde \picorsivo_{(x_0,0)}(\widetilde \Omega)$. By Theorem \ref{maintheorem} 
 $$\sup_{\widetilde K} \widetilde u \le C\  \widetilde u(x_0, 0 )\,$$ and the conclusion follows immediately.
 
 \end{proof}

\section{Examples}

In this Section we give two examples of operators for which we give Harnack-type inequalities that, to our knowledge, 
are new. In general, the main step in the application of our Theorem \ref{maintheorem} is the 
characterization of the propagation set $\picorsivo(x_0,\Omega)$ of the operator $\elle$. We recall that the Control 
Theory provides us with several tools useful for this problem. We refer, for example, to the  book \cite[Chapter 
8]{agrachev_sachkov} by Agrachev and Sachkov.

\subsection{A Harnack inequality for the stationary Mumford operator}\mbox{}\\
We consider the operator $\elle =\partial_{x_1}^2 + \sin(x_1) \partial_{x_2} + \cos(x_1) \partial_{x_3}$ in the set: 
\begin{equation} \label{eq-omega-Mum}
  \Omega= ]-a,a[ \times B(0,r)\subseteq \erre \times \erre^2.
\end{equation}
$x_1\in ]-a,a[$ where $a>\pi$, and $(x_2,x_3)\in B(0,r)$, the euclidean ball centered at $0$ with radius $r>0$.
This operator has been introduced by Mumford \cite{mumford_1994} in the study of computer vision problems.
The relevant Harnack inequality of Theorem \ref{maintheorem} takes the following form:
 \begin{theorem} Let $\Omega$ be the set introduced in \eqref{eq-omega-Mum}, with $a>\pi$. For every compact set 
$K\subset \Omega$ there exists a positive constant $C=C(K,\Omega,\elle)$ such that 
$$\sup_K u \le C u(0),$$
for every non-negative solution $u$ of
$$
\partial_{x_1}^2 u + \sin(x_1) \partial_{x_2} u + \cos(x_1) \partial_{x_3}u=0 \quad \text{in} \quad \Omega.
$$
\end{theorem}
\begin{proof}
In view of Theorem \ref{maintheorem}, we need only to prove that in this case the {\it propagation set} 
$\picorsivo(0,\Omega)$ agrees with $\Omega$. With this aim, we fix any point $z=(z_1,z_2,z_3)$ in $\Omega$, and we 
construct a $\elle$-{\it propagation path} steering $0$ to $z$. Note that, in our case, the vector fields defined in 
\eqref{eq-XY} are 
$$
  X=\partial_{x_1} \qquad \text{and} \qquad Y=  \sin(x_1) \partial_{x_2} + \cos(x_1) \partial_{x_3}.
$$
We connect $0$ and $z$ by a path $\gamma: [0,T] \to \Omega$ such that $\gamma'(t) = \pm X(\gamma(t))$ in the first 
interval $[0, t_1]$, then $\gamma'(t) = Y(\gamma(t))$ in the second interval $[t_1, t_2]$, and $\gamma'(t) = \pm 
X(\gamma(t))$ in the third interval $[t_2, T]$, for $t_1, t_2, T$ such that $0 \le t_1 \le t_2 \le T$ chosen as 
follows.

We set $t^* = \arg(z_2,z_3) \in ]- \pi, \pi] \subset ]-a,a[$, and we choose $t_1 := |t^*|$. If $t^* >0$, the function 
$\gamma(t) = (t,0,0)$ is a solution of $\gamma'(t) = X$, for $t \in [0,t_1]$, $\gamma(0)=0$. If $t^* <0$ we consider 
$\gamma(t) = (-t,0,0)$. In both cases, we have that $\gamma'(t) = \pm X(\gamma(t))$. If $t^* = 0$ we simply skip this 
step.

We next set $t_2 = t_1 + \sqrt{z_2^2 +z_3^2}$, and we choose $\gamma$ such that $\gamma'(t) = Y(\gamma(t))$ for $t_1 < 
t < t_2$.  Also in this case, if 
$(z_2,z_3)=(0,0)$, we skip this step. 
We conclude the construction of $\gamma$ by choosing $s^* = z_1 - t^*$, $T = t_2 + |s^*|$ and following the same method 
used in the first step. The path $\gamma$ then writes as follows.

\begin{eqnarray*}
\gamma(t)= \left\{\begin{array}{ccc}&(\pm t,0,0) & \quad \mbox{ if }\qquad 0\le t \le t_1, \\ \\
&(t^\ast , (t - t_1)  \cos t^*, (t - t_1)  \sin t^*)& \mbox{ if }\quad t_1 \le  t \le t_2,  \\ \\
& (t^\ast  \pm (t - t_2),z_2, z_3)& \mbox{ if }\quad  t_2 \le t \le T. 
\end{array} \right.
\end{eqnarray*}
\end{proof}

\begin{remark} The above construction can be reproduced to translated cylinders
\begin{equation*} 
  \Omega_y= ]y_1 -a,y_1+ a[ \times B((y_2,y_3),r)\subseteq \erre \times \erre^2,
\end{equation*}
for every $y = (y_1,y_2,y_3) \in \erre^3$. We find $\picorsivo((y_1,y_2,y_3),\Omega_y) = \Omega_y$.

We point out that, on the other hand, the geometry of the propagation set $\picorsivo(0,\Omega)$ changes completely as 
the width of the interval $]-a,a[$ is smaller than $2 \pi$. For instance, if we consider the set
\begin{equation*} 
  \widetilde \Omega = ]-\pi/2, \pi/2[ \times B(0,r)\subseteq \erre \times \erre^2,
\end{equation*}
we easily see that $\picorsivo(0,\widetilde \Omega) = \widetilde \Omega \cap \big\{ x_3 > 0 \big \}$. This fact is in accordance with the invariance of the operator $\elle$ with respect to the following left translation introduced in \cite{bonfiglioli_lanconelli_matrix}. Denote $x=(t,z), y=(s,w) \in \erre \times \mathbb{C}$. then 
$$ x\circ y:= (t+s, z + w e^{it}).$$ 

\end{remark}
\subsection{A Harnack inequality for a degenerate Ornstein Uhlenbeck operator}\mbox{}\\
We consider the operator $\elle =\partial_{x_1}^2 + x_1 \partial_{x_2}$ in the set 
\begin{equation} \label{eq-omega-OU}
\Omega= ]-a,a[ \times ]-b,b[
\end{equation}
for some positive $a$ and $b$.

As in the case of Mumford operator, Theorem \ref{maintheorem} gives an \emph{elliptic} Harnack inequality.
\begin{theorem} Let $\Omega$ be the set introduced in \eqref{eq-omega-OU}. For every compact set 
$K\subset \Omega$ there exists a positive constant $C=C(K,\Omega,\elle)$ such that 
$$\sup_K u \le C u(0),$$
for every non-negative solution $u$ of
$$
\partial_{x_1}^2 u + x_1 \partial_{x_2} u=0 \quad \text{in} \quad \Omega.
$$
\end{theorem}
\begin{proof}
We prove that, also in this case, the {\it propagation set} $\picorsivo(0,\Omega)$ agrees with $\Omega$. The vector 
fields defined in \eqref{eq-XY} are 
$$
  X=\partial_{x_1} \qquad \text{and} \qquad Y=  x_1 \partial_{x_2}.
$$
We choose an integral curve $\gamma$ such that $\gamma'(t) = \pm X(\gamma(t))$ in some intervals. A curve like that 
writes as $\gamma(t) = (\widetilde x_1 \pm t, \widetilde x_2)$. In particular, we will use the field $X$ to increase or 
decrease the first coordinate $x_1$. In some other intervals we choose $\gamma'(t) = Y(\gamma(t))$. Such a curve 
writes as $\gamma(t) = (\widetilde x_1, \widetilde x_2 + \widetilde x_1 t)$. In this case, we rely on the sign of 
$\widetilde x_1$ to increase or decrease the second component $x_2$. We prefer not to give the details of the construction and 
to refer to the following figure.

%
\vspace{0.5cm}\begin{center}
\scalebox{.80}
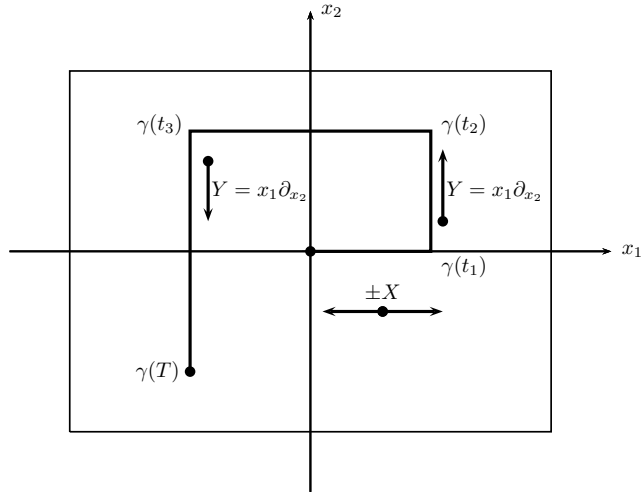
{
\begin{pspicture}(-5.5,-4)(5.5,4)
  \psaxes[Dx=6, Dy=5]{->}(0,0)(-5,-4)(5,4)
  \psline(-4,-3)(4,-3)(4,3)(-4,3)(-4,-3)
  \psline[linewidth=1.5pt]{*-*}(0,0)(2,0)(2,2)(-2,2)(-2,-2)
  \uput[0](5,0){$x_1$}
  \uput[0](0,4){$x_2$}
  \uput[-8](2,-.2){$\gamma(t_1)$}
  \uput[0](2,2.1){$\gamma(t_2)$}
  \uput[225](-2,2.4){$\gamma(t_3)$}
  \uput[180](-2,-2){$\gamma(T)$}
  \psline[linewidth=1.5pt]{<-*}(.2,-1)(1.2,-1)
  \psline[linewidth=1.5pt]{*->}(1.2,-1)(2.2,-1)
  \uput[90](1.2,-1){$\pm X$}
  \psline[linewidth=1.5pt]{*->}(-1.7,1.5)(-1.7,.5)
  \uput[0](-1.7,1){$\! \! Y = x_1 \partial_{x_2}$}
  \psline[linewidth=1.5pt]{*->}(2.2,.5)(2.2,1.7)
  \uput[0](2.2,1){$\! \! Y = x_1 \partial_{x_2}$}
\end{pspicture}}\end{center}
\end{proof}
%

%

\begin{remark} The above result fails as  
\begin{equation*} 
  \Omega= ]a_1,a_2[ \times ]-b,b[\subseteq  \erre^2,
\end{equation*}
and $a_1$ and $a_2$ have the same sign. In particular, if $a_1$ and $a_2$ are both positive, and we consider $x_0 = 
\left(\frac{a_1 + a_2}{2},0\right)$, we have $\picorsivo(x_0,\Omega) = \Omega \cap \big\{ x_2 > 0 \big \}$. On the 
contrary, if $a_1$ and $a_2$ are both negative, we have $\picorsivo(x_0,\Omega) = \Omega \cap \big\{ x_2 < 0 \big \}$.
\end{remark}

\section*{Acknowledgments}
%
%

The authors have been partially supported by the Gruppo Nazionale per l'Analisi Matematica, la Probabilit\`a e le
loro Applicazioni (GNAMPA) of the Istituto Nazionale di Alta Matematica (INdAM).

\end{document}